\documentclass[11pt]{amsart}
\usepackage{bbm}
\usepackage{amsmath}  
\usepackage{amsthm}     
\usepackage{enumerate} 
\usepackage{amssymb}
\usepackage{physics}
\usepackage{mathtools}
\usepackage{hyperref}
\usepackage[margin=1.0in]{geometry} 
\usepackage{comment}
\usepackage{xcolor} 
\DeclareMathOperator*{\argmin}{argmin}

\newtheorem{remark}{Remark}
\newtheorem{prop}{Proposition}
\newtheorem{thm}{Theorem}
\newtheorem{lemma}{Lemma}
\newtheorem{assume}{Assumption}
\newtheorem{definition}{Definition}

\newtheorem{cor}{Corollary}
\numberwithin{remark}{section}
\numberwithin{prop}{section}
\numberwithin{thm}{section}
\numberwithin{lemma}{section}
\numberwithin{definition}{section}
\numberwithin{equation}{section}
\numberwithin{assume}{section}
\numberwithin{example}{section}
\numberwithin{cor}{section}

\def\W{\mathcal{W}}
\def\E{\mathbb{E}}
\def\P{\mathbb{P}}
\def\cP{\mathcal{P}}
\def\R{\mathbb{R}}
\def\k{\kappa}
\def\a{\alpha}

\def\e{\epsilon}
\def\pa{\partial}
\def\cH{\mathcal{H}}
\def\cL{\mathcal{L}}

\def\ld{\lambda}
\def\ol{\overline}
\def\e{\epsilon}
\def\wt{\widetilde}

\makeatletter
\@namedef{subjclassname@2020}{%
  \textup{2020} Mathematics Subject Classification}

\title[]{Solvability of Infinite horizon McKean-Vlasov FBSDEs in Mean Field Control Problems and Games}
\author{Erhan Bayraktar}\thanks{E. Bayraktar is partially supported by the National Science Foundation under grant DMS-2106556 and by
the Susan M. Smith chair.}
\address{Department of Mathematics, University of Michigan}
\email{erhan@umich.edu}
\author{Xin Zhang} 
\address{Department of Mathematics, University of Michigan}
\email{zxmars@umich.edu}
\date{\today}
\keywords{McKean-Vlasov FBSDE, infinite horizon, mean field game, mean field type control.}
\subjclass[2020]{Primary: 91A16, 49N80. }

\begin{document}

\begin{abstract}
In this paper, we show existence and uniqueness of solutions of the infinite horizon McKean-Vlasov FBSDEs using two different methods, which lead to two different sets of assumptions. We use these results to solve the infinite horizon mean field type control problems and mean field games.
\end{abstract}

\maketitle

\section{Introduction}
Motivated by infinite horizon mean field control and mean field game, in this paper we establish existence and uniqueness of solutions to an infinite horizon McKean-Vlasov FBSDE 
\begin{align}\label{eq:gMFB}
\begin{cases}
dX_t=B(t,X_t,Y_t, \cL(X_t,Y_t))\, dt + \sigma \,dW_t, \\
 dY_t=-F(t,X_t,Y_t, \cL(X_t,Y_t)) \, dt+Z_t \, dW_t, \quad \forall {t \geq 0}, \\
 X_0=\xi,
\end{cases}
\end{align}
where $(W_t)$ is a Brownian motion on a stochastic basis $(\Omega, \mathcal{F}, \mathbb{F}, \P)$, $B,F: \Omega \times \R_+ \times \R^2 \times \cP_2(\R^2) \to \R$ are two progressively measurable functions, and $\xi$ is an $\mathcal{F}_0$-measurable square integrable random variable. Compared with finite horizon FBSDEs, in \eqref{eq:gMFB} no terminal conditions are required. Instead, for the well-posedness we specify the solution space which determines asymptotic behavior of the processes. Due to our interest in infinite horizon discounted problems, we look for solutions $(X_t,Y_t,Z_t)$ to \eqref{eq:gMFB} in $L^2_K(0,\infty, \R^3)$, where $K \in \R$ and $L^2_K(0,\infty, \R^3)$ is the Hilbert space of all $\R^2$-valued adapted stochastic process $(v_t)$ such that 
\begin{align*}
\E \left[ \int_0^{\infty} e^{-Kt} |v_t|^2 \, dt  \right]< +\infty. 
\end{align*}
Using methods of \cite{MR1730617} and \cite{MR1701517,MR2462676}, we show that there exists a unique solution $(X_t,Y_t) \in L^2_K(0,\infty,\R^2)$ to \eqref{eq:gMFB} under two sets of assumptions. As applications, we solve the corresponding infinite horizon McKean-Vlasov FBSDEs of mean field type control and mean field game problems.

The study of mean field games was initiated independently by Lasry, Lions (see  \cite{MR2269875}, \cite{LASRY2006679}, \cite{MR2295621}) and Caines, Huang, Malham\'{e} (see  \cite{4303232}, \cite{MR2346927}), which is an analysis of 
limit models for symmetric weakly interacting $N+1$-player differential games.
Since then, mean field game has been an active research area. We refer the readers to \cite{2020arXiv201204845B,2019arXiv191206701B,MR3860894,MR4013871} for the study of finite state mean field games, to \cite{MR4127851,MR4083905,MR3981375,MR4046528} for uniqueness of mean field game solutions, and to  \cite{MR3752669,MR3753660} for a nice survey. Also, inspired by the surge of interest in optimal control, several works have been published for the analysis of mean field type control, which includes the distribution of controlled states in coefficients; see e.g. \cite{MR2784835,MR3134900,MR3045029}.

The investigation of BSDEs was pioneered by Pardoux and Peng \cite{10.1007/BFb0007334,MR1037747} in the early 90s, which is now a standard tool in stochastic optimization problems (see e.g. \cite{MR3629171,MR2533355}). Applying Pontryagin's maximum principle, both mean field game and mean field type control can be studied using McKean-Vlasov FBSDEs; see e.g. \cite{MR3091726,MR3045029,doi:10.1142/S0219493721500362}. For analysis of FBSDE, we refer to a common reference \cite{MR1704232}. 

The linear quadratic model for infinite horizon mean field game and mean field type control have been studied in \cite{cardaliaguet2019,4303232,MR3377926} using HJB-FP equations and in \cite{MR3938473} using martingale method respectively. \cite{MR3377926} provided the exact stationary solution to linear quadratic infinite horizon mean field games. We also refer to \cite{cardaliaguet2019,2021arXiv210109965C} for the PDE analysis of long time behavior of mean field game. For the best of our knowledge, this paper is the first to investigate infinite horizon mean field game and mean field type control problems using FBSDE techniques. 

The result of the paper is organized as follows. In Section~\ref{sec2}, we prove the existence and uniqueness of solutions to \eqref{eq:gMFB} under two sets of assumptions; see Theorems~\ref{thm1} and \ref{thm2}. In Section~\ref{sec3}, as an application, we solve the infinite horizon mean field type control problems and games. In Section~\ref{sec4}, we analyze the particular case of linear quadratic models. 

In this rest of this section we will list some frequently used notation.

\noindent {\bf{Notation.}} Denote by $\cP_2(\R^n)$ the space of random variables in $\R^n$ with finite second moment endowed with the Wasserstein 2-metric $\W_2$. For any $\R^n$, define $\delta_0$ to be the Dirac measure at the origin, and for any random variable $X$, denote by $\cL(X)$ the law of $X$.

\section{Solutions to infinite horizon McKean-Vlasov FBSDEs}\label{sec2}

In this section, we establish the existence and uniqueness of the infinite horizon McKean-Vlasov FBSDE \eqref{eq:gMFB} under two sets of assumptions. For any $(v_t) \in L^2_K(0,\infty, \R^n)$, we define the exponentially weighted $L^2$ norm 
\begin{align*}
||v||^2_K:=\E \left[ \int_0^{\infty} e^{-Kt} |v_t|^2 \, dt \right].
\end{align*}
For simplicity, we only solve \eqref{eq:gMFB} for one dimensional $(X_t,Y_t,Z_t)$, but our results can be easily generalized to multidimensional case.
\subsection{Continuity method}
As in \cite{MR1730617}, we study the following family of infinite horizon FBSDEs parametrized by $\ld \in [0,1]$,
\begin{align}\label{eq:FFB}
\begin{cases}
&d X^{\ld}_t= (\ld B(t,X^{\ld}_t, Y^{\ld}_t, \cL(X^{\ld}_t,Y^{\ld}_t))-\k (1-\ld)Y^{\ld}_t + \phi(t) ) \, dt +\sigma \,dW_t, \\
&d Y^{\ld}_t=-(\ld F(t, X^{\ld}_t, Y^{\ld}_t, \cL(X^{\ld}_t,Y^{\ld}_t))+\k (1-\ld) X^{\ld}_t+\psi(t) ) \, dt + Z^{\ld}_t \, dW_t, \\
&X^{\ld}_0=\xi, \quad (X^{\ld}_t,Y^{\ld}_t,Z^{\ld}_t) \in L_{K}^2(0,\infty, \R^3),
\end{cases}
\end{align}
where $\phi, \psi$ are two arbitrary processes in $L^2_{K}(0,\infty, \R)$ and $\k$ is a positive constant to be determined below in Assumption~\ref{assume1}.
Note that when $\ld=1$, $\phi\equiv 0$, $\psi \equiv 0$, \eqref{eq:FFB} becomes \eqref{eq:gMFB}, and when $\ld=0$, \eqref{eq:FFB} becomes 
\begin{align}\label{eq:lambda0}
\begin{cases}
&dX^0_t= (-\k Y_t^0+\phi(t)) \, dt +\sigma\, dW_t, \\
&dY^0_t=-(\k X_t^0+\psi(t)) \, dt+Z_t^0 \, dW_t, \\
&X^0_0=\xi.
\end{cases}
\end{align}

\begin{lemma}\label{sec3:lem1}
Assume that $0 <K <2\k$. For any $\phi, \psi \in L^2_{K}(0,\infty, \R)$, there exists a unique solution $(X^0,Y^0,Z^0) \in L^2_{K}(0,\infty, \R^3)$ to \eqref{eq:lambda0}. 
\end{lemma}
\begin{proof}
The argument is almost the same as \cite[Lemma 2]{MR1730617}, and we repeat it here for readers' convenience. Let us consider the following infinite horizon BSDE,
\begin{align*}
d P_t = -(- \k P_t +\phi(t)+\psi(t) ) \, dt +(Q_t-\sigma) \, dW_t, \quad \forall t \geq 0. 
\end{align*}
Applying \cite[Theorem 4]{MR1730617} with the fact that $K - 2\k<0$, the above equation has a unique solution $(P,Q) \in L^2_{K}(0,\infty, \R)$. Then we consider the following SDE, 
\begin{align*}
d X_t = ( - \k X_t -\k P_t + \phi(t)) \, dt + \sigma \, dW_t,  \quad X_0=\xi. 
\end{align*}
Since $P, Q, \phi \in L^2_K(0,\infty,\R)$, it can be easily seen that the above equation has a unique solution over arbitrary finite horizon $[0,T]$. Therefore, it remains to show that $X \in L^2_K(0,\infty, \R)$. Applying It\^{o}'s formula to $e^{-Kt} |X_t|^2$, it follows that 
\begin{align*}
& \E[e^{-KT}|X_T|^2 ]-\E[\xi^2] \\
&= \E \left[ \int_0^T (-K-2\k)e^{-Kt} |X_t|^2+2 e^{-Kt}X_t \cdot (-\k P_t +\phi(t)) ) \, dt \right]  + \E \left[ \int_0^T e^{-Kt} \sigma^2 \, dt \right].
\end{align*}
Choose a positive $\e$  such that $-K-2\k +\e <0$. Using the inequality 
\begin{align*}
2 e^{-Kt} X_t \cdot (-\k P_t +\phi(t)) \leq \e e^{-Kt}|X_t|^2+ \frac{e^{-Kt}}{\e}(-\k P_t+\phi(t))^2,
\end{align*}
we easily obtain that 
\begin{align*}
\E[e^{-KT}|X_T|^2 ]-\E[\xi^2] \leq \E \left[\int_0^T (-K- 2 \k +\e) e^{-Kt}|X_t|^2 \, dt \right]+C_{\e},
\end{align*}
where $C_{\e}$ is a constant that only depends on $\e$ and $\Vert P \Vert^2_K,\Vert Q \Vert^2_K,\Vert \phi \Vert^2_K$. Clearly, it is equivalent to 
\begin{align*}
\E[e^{-KT}|X_T|^2 ]+(2 \k +K-\e)\E \left[\int_0^T  e^{-Kt}|X_t|^2 \, dt \right] \leq \E[\xi^2]+C_{\e}.
\end{align*}
Letting $T \to \infty$ in the above inequality and noting that $2\k+K-\e>0$, we conclude that $X \in L^2_K(0,\infty, \R)$. It can be easily verified that $(X^0,Y^0,Z^0)=(X,X+P,Q) \in L^2_K (0,\infty, \R^3)$ is a solution to \eqref{eq:lambda0}. The uniqueness can be proved in a similar way as in Theorem~\ref{thm1}.

\end{proof}

\begin{assume}\label{assume1}
(i) There exists a positive constant $l$ such that for any $x,x',y,y' \in \R$, $ m, m' \in \mathcal{P}_2(\R^2)$
\begin{align*}
  | B(t,x,y,m)-& B(t,x',y',m')|+|F(t,x,y,m)-F(t,x',y',m')| \\
 & \leq l(|x-x'|+|y-y'|+\W_2(m,m') ) \quad \text{a.s.}
\end{align*}
(ii) There exist constants $0<K<2\k $ such that for any $t  \geq 0$ and any square integrable random variables $X, Y, X', Y'$
\begin{align*}
\E \left[{-K} \hat{X} \hat{Y}- \hat{X} (F(t,U)-F(t,U') )+\hat{Y} (B(t,U)-B(t,U') )\right] \leq -\k \E \left[ \hat{X}^2+ \hat{Y}^2 \right],
\end{align*}
where $\hat{X}=X-X', \hat{Y}=Y-Y'$ and $U=(X, Y , \cL(X,Y)), U'=(X', Y', \cL(X',Y')).$
\end{assume}

\begin{prop}\label{sec3:prop1}
Suppose $\ld_0 \in [0,1)$ and for any $\mathcal{F}_0$-measurable square integrable random variable $\xi$, $\phi, \psi \in L_K^2(0,\infty, \R)$, \eqref{eq:FFB} has a unique solution $(X^{\ld_0}, Y^{\ld_0}, Z^{\ld_0})$ in $L_K^2(0,\infty, \R^3)$. Then under Assumption~\ref{assume1}, the FBSDE \eqref{eq:FFB} has a unique solution $(X^{\ld_0+\delta}, Y^{\ld_0+\delta}, Z^{\ld_0+\delta})$ in $L_K^2(0,\infty, \R^3)$ for any $\delta \leq \frac{2\k}{3\k+ 12l}$, $\phi, \psi \in L_K^2(0,\infty,\R^2)$
\end{prop}
\begin{proof}
For any pair $(x,y) \in L^2_K(0,\infty, \R^2)$ such that $x_0=\xi$, according to our hypothesis, there exists a unique solution $(X,Y,Z)$ to the following equation 
\begin{align*}
\begin{cases}
&dX_t=(\ld_0 B(t, X_t,Y_t, M_t)-\k(1-\ld_0)Y_t+\delta (B(t,x_t,y_t,m_t)+\k y_t)+\phi(t)) \, dt  +\sigma \, dW_t ,\\
&dY_t=-(\ld_0 F(t,X_t,Y_t,M_t)+\k(1-\ld_0)X_t+\delta (F(t,x_t,y_t,m_t)-\k x_t)+\psi(t))\, dt +Z_t \, dW_t, \\
&X_0=\xi,
\end{cases}
\end{align*}
where $m_t:=\cL(x_t,y_t)$ and $M_t:=\cL(X_t,Y_t)$.
We define a map $\Phi$ via
\begin{align*}
\Phi: (x,y) \mapsto (X,Y). 
\end{align*}
Then a fixed point of $\Phi$ is a solution to \eqref{eq:FFB} with parameter $\ld_0+\delta$. Let us prove that $\Phi$ is actually a contraction.

Take another $(x',y')$ and its image $(X',Y')$ under $\Phi$. Denote $u_t=(x_t,y_t,m_t), U_t=(X_t,Y_t,M_t)$, and $\hat{x}_t=x_t -x'_t, \hat{y}_t=y_t-y'_t$ and similarly $\hat{X}_t, \hat{Y}_t$. Since $\hat{X}, \hat{Y} \in L_K^2(0,\infty, \R)$, there exists an increasing sequence of $T_i$ such that $\lim\limits_{i \to \infty} T_i=\infty$ and 
\begin{align*}
\lim\limits_{i \to \infty} \E\left[ e^{-KT_i} \hat{X}_{T_i} \hat{Y}_{T_i} \right]=0.
\end{align*}
By It\^{o}'s formula, it can be easily seen that
\begin{align}\label{eq:ito}
\E\left[ e^{-KT_i} \hat{X}_{T_i} \hat{Y}_{T_i} \right]=& \ld_0\E\left[\int_0^{T_i} e^{-Kt} \left( {-K} \hat{X}_t \hat{Y}_t- \hat{X}_t (F(t,U_t)-F(t,U'_t) )+\hat{Y}_t (B(t,U_t)-B(t,U'_t) ) \right) dt  \right] \notag \\
&-\k(1-\ld_0)  \E \left[ \int_0^{T_i} e^{-Kt} \left( \hat{X}_t^2+ \hat{Y}_t^2 \right) dt\right] -(K-\ld_0{K}) \E \left[ \int_0^{T_i} e^{-Kt} \hat{X}_t \hat{Y}_t \, dt \right]  \notag \\
& + \k \delta \E\left[ \int_0^{T_i} e^{-Kt} \left( \hat{X_t} \hat{x}_t + \hat{Y_t} \hat{y_t}  \right)  dt \right] \notag \\
&+ \delta \E \left[\int_0^{T_i} e^{-Kt} \left( -\hat{X}_t (F(t,u_t)-F(t,u'_t))+\hat{Y}_t ( B(t,u_t)-B(t,u'_t)) \right) dt \right] . 
\end{align}
According to Assumption~\ref{assume1} (ii), it holds that 
\begin{align}\label{eq:mainassump}
\E \left[{-K} \hat{X}_t \hat{Y}_t- \hat{X}_t (F(t,U_t)-F(t,U'_t) )+\hat{Y}_t (B(t,U_t)-B(t,U'_t) )\right] \leq -\k \E \left[ \hat{X}_t^2+ \hat{Y}_t^2 \right].
\end{align}

Therefore by Assumption~\ref{assume1} (i) and the fact that $$\W_2( \cL(x_t,y_t), \cL(x'_t,y'_t)) \leq \sqrt{\E[|x_t-x'_t|^2 ]} + \sqrt{\E[|y_t-y'_t|^2 ]},$$ it can be easily deduced from \eqref{eq:ito}  
\begin{align*}
\E\left[ e^{-KT_i} \hat{X}_{T_i} \hat{Y}_{T_i} \right] \leq &-\left(\k-K/2-\frac{k\delta +4l \delta}{2} \right)\E \left[ \int_0^{T_i} e^{-Kt} \left( \hat{X}_t^2+ \hat{Y}_t^2 \right) dt\right]\\
&+ \frac{\k \delta +4 l \delta }{2 } \E \left[\int_0^{T_i} e^{-Kt}\left( \hat{x}_t^2+ \hat{y}_t^2  \right) dt \right].
\end{align*}
Letting $i \to \infty$ and choosing $\delta \leq \frac{2\k}{3\k+ 12l}$, we actually obtain that 
\begin{align*}
\E \left[\int_0^{\infty} e^{-Kt} \left( \hat{X}^2_t + \hat{Y}^2_t \right) dt  \right] \leq \frac{1}{2} \E \left[\int_0^{\infty} e^{-Kt} \left( \hat{x}^2_t + \hat{y}^2_t \right) dt  \right],
\end{align*}
and therefore $\Phi$ is a contraction. 
\end{proof}

\begin{thm}\label{thm1}
Under Assumption~\ref{assume1}, for each $\mathcal{F}_0$-measurable square integrable random variable $\xi$, \eqref{eq:gMFB} has a unique solution in $L^2_{K}(0,\infty,\R^3)$. 
\end{thm}

\begin{proof}
By Lemma~\ref{sec3:lem1}, for any  $\phi, \psi \in L^2_K(0,\infty, \R)$, there exists a solution in $L^2_K(0,\infty, \R)$ to \eqref{eq:FFB} with $\lambda=0$. Then according to Proposition~\ref{sec3:prop1}, for any $\phi,\psi \in L^2_K(0,\infty,\R)$ there exists a  solution to \eqref{eq:FFB} with $\lambda=\delta_0$. Repeating this process for $\lceil \frac{1}{\delta_0}\rceil$ many times, we conclude that there exists a solution to \eqref{eq:FFB} with $\lambda=1$. In particular, letting $\phi \equiv 0, \psi \equiv 0$, we get a solution to \eqref{eq:gMFB}. 

For the uniqueness, suppose there exist two solutions $(X,Y,Z), (X',Y',Z') \in L^2_K(0,\infty, \R^3)$ to \eqref{eq:gMFB}, and denote $\hat{X}=X-X', \hat{Y}=Y-Y', \hat{Z}=Z-Z'$. There exists a sequence of $T_i \to \infty$ such that $\E\left[ e^{-KT_i} \hat{X}_{T_i} \hat{Y}_{T_i}\right] \to 0$. By It\^{o}'s formula and Assumption~\ref{assume1}, we have that 
\begin{align*}
\E\left[ e^{-KT_i} \hat{X}_{T_i} \hat{Y}_{T_i} \right]=&\E\left[\int_0^{T_i} e^{-Kt} \left( {-K} \hat{X}_t \hat{Y}_t- \hat{X}_t (F(t,U_t)-F(t,U'_t) )+\hat{Y}_t (B(t,U_t)-B(t,U'_t) ) \right) dt  \right] \\
\leq & - (\k-K/2) \E\left[\int_0^{T_i} e^{Kt} \left(\hat{X}_t^2+\hat{Y}_t^2 \right) dt \right].
\end{align*}
Letting $T_i \to \infty$, we conclude that $\Vert \hat{X} \Vert^2_K=\Vert \hat{Y} \Vert^2_K =0$, and hence complete the proof. 
\end{proof}

\subsection{Fixed point argument}
We prove the existence of solution to \eqref{eq:gMFB} under another monotonicity condition, which in the spirit of \cite{MR1701517},\cite{MR2462676}. 
The main idea is as follows. Take any process $(x_t) \in L^2_{K}(0,\infty, \R)$ such that $x_0=\xi$. Using \cite[Theorem 4.1]{MR1695013},   there exists a unique solution $(\ol{y}_t,\ol{z}_t)$ to the following infinite horizon BSDE 
\begin{align}\label{eq:BSDE}
d\ol{y}_t=-F(t,x_t,\ol{y}_t,\cL(x_t,\ol{y}_t)) \,dt + \ol{z}_t \, dW_t, \quad \forall t \geq 0.
\end{align}
And then we show that there exists a unique solution to the forward McKean Vlasov SDE 
\begin{align}\label{eq:SDE}
\begin{cases}
dX_t=B(t,X_t,\ol{y}_t,\cL(X_t,\ol{y}_t))\, dt + \sigma \, dW_t, \\
X_0=\xi,
\end{cases}
\end{align}
and hence we construct a mapping which sends $(x_t)$ to $(X_t)$. We will prove that this mapping is a contraction, and hence its unique fixed point is the unique solution to \eqref{eq:gMFB}. First we present the main assumption of this subsection.

\begin{assume}\label{assume2}
(i) There exists some constants $\k_1,\k_2$ such that for any $t \in \R_+$, $x,x',y,y' \in \R$, $m \in \cP_2(\R^2)$
\begin{align*}
 (y-y')(F(t,x,y,m)-F(t,x,y',m) \leq -\k_1 |y-y'|^2 \quad \text{a.s.}, \\ 
 (x-x') (B(t,x,y,m)-B(t,x',y,m) )\leq -\k_2 |x-x'|^2 \quad \text{a.s.} 
 \end{align*}
(ii) $F(t,x,y,m), B(t,x,y,m)$ are Lipschitz in $(x,y,m)$. There exist some positive constant $l_1,l_2$ such that for any $t \in \R_+$, $x,x',y,y' \in \R$, $m,m' \in \cP_2(\R^2)$
\begin{align*}
|F(t,x,y,m)-F(t,x',y,m')| \leq l_1(|x-x'|+\W_2( m ,m') ) \quad \text{a.s.},\\
|B(t, x, y, m)-B(t,x,y', m')| \leq l_2 (|y-y'|+\W_2(m, m')) \quad \text{a.s.}
\end{align*}
(iii) There exist some positive constants $\e_1, \e_2$ and positive constant $K$ such that 
\begin{align*}
-2\k_2 +2 l_2 +2l_2 \e_2<K<2 \k_1 - 2l_1-2l_1 \e_1 ,
\end{align*}
and also 
\begin{align*}
 4l_1 l_2 \leq \e_1 \e_2 (-K +2 \k_1-2l_1 -2l_1 \e_1)(K+2 \k_2- 2 l_2 -2l_2 \e_2) .\end{align*}
(iv) $\Vert F(\cdot, 0,0,\delta_0) \Vert^2_{K} + \Vert B(\cdot, 0,0,\delta_0) \Vert^2_{K} < +\infty$.
\end{assume}

\begin{lemma}
Under Assumption~\ref{assume2}, for any $(x_t) \in L^2_K(0,\infty,\R)$ there exists a unique solution $(\ol{y},\ol{z})$ to \eqref{eq:BSDE} such that $(\ol{y},\ol{z})  \in  L^2_{K}(0, \infty, \R^2)$.
\end{lemma}
\begin{proof}
According to \cite[Theorem 4.1]{MR1695013}, for any $(y_t) \in L^2_K(0, \infty,\R)$, there exists a unique solution $(\ol{y}_t,\ol{z}_t) \in L^2_K(0,\infty, \R^2)$ to the infinite horizon BSDE 
\begin{align}\label{eq:BSDE'}
d\ol{y}_t=-F(t,x_t,\ol{y}_t,\cL(x_t,{y}_t)) \,dt + \ol{z}_t \, dW_t, \quad \forall t \geq 0.
\end{align}
Therefore it suffices to show that $(y_t) \mapsto (\ol{y_t})$ is a contraction on $L^2_K(0, \infty,\R)$. Take any $(y_t), (y_t') \in L^2_K(0,\infty, \R)$, and denote by $(\ol{y}_t), (\ol{y}_t')$ their corresponding solutions to \eqref{eq:BSDE'}.

From It\^{o}'s formula, one can easily deduce that 
\begin{align}\label{lem2:2}
-K e^{-K t} |\ol{y}_t-\ol{y}_t'|^2 \,dt +e^{-K t} |\ol{z}_t-\ol{z}_t'|^2 \, dt= d e^{-K t} |\ol{y}_t-\ol{y}_t'|^2-2 e^{-K t} (\ol{y}_t-\ol{y}_t') \, d (\ol{y}_t-\ol{y}_t')
\end{align}
Since $\ol{y}, \ol{y}' \in L^2_{K}(0,\infty, \R)$, there exists a sequence of $T_i \to \infty$ such that $\E\left[e^{-K T_i} |\ol{y}_{T_i}-\ol{y}_{T_i}'|^2 \right] \to 0$.
Integrating \eqref{lem2:2} over interval $[0,T_i]$, taking expectation, and letting $T_i \to \infty$, we obtain that 
\begin{align*}
\E& \left[ \int_0^{\infty} -K e^{-K t} |\ol{y}_t-\ol{y}_t'|^2  + e^{-K t} | \ol{z}_t-\ol{z}_t'|^2 \,dt  \right]= -\E\left[|\ol{y}_0-\ol{y}_0'|^2 \right] \notag \\
& \quad \quad \quad \quad +\E \left[\int_0^{\infty} 2 e^{-K t} (\ol{y}_t-\ol{y}_t') \left( F(t,x_t,\ol{y}_t, \cL(x_t,y_t))-F(t,x_t,\ol{y}_t', \cL(x_t, y_t'))\right)  dt \right].
\end{align*}
For the second term on the right hand side, we have that 
\begin{align*}
&2 e^{-K t} (\ol{y}_t - \ol{y}_t')\left( F(t,x_t,\ol{y}_t, \cL(x_t,y_t) )-F(t,x_t,\ol{y}_t', \cL(x_t,y_t') ) \right) \\
& \quad \quad \leq 2 e^{-K t} (\ol{y}_t - \ol{y}_t')\left( F(t,x_t,\ol{y}_t, \cL(x_t,y_t) )-F(t,x_t,\ol{y}_t', \cL(x_t,y_t) ) \right) \\
& \quad \quad \ \ \ + 2 e^{-K t} (\ol{y}_t - \ol{y}_t')\left( F(t,x_t,\ol{y}_t', \cL(x_t,y_t) )-F(t,x_t,\ol{y}_t', \cL(x_t,y_t') ) \right) \\
& \quad \quad  \leq -2  \k_1   e^{-K t} |\ol{y}_t - \ol{y}_t'|^2 + 2  e^{-K t} |\ol{y}_t-\ol{y}_t'| \left( \W_2(\cL(x_t,y_t), \cL(x_t,y_t')) \right) 
\end{align*}
Together with $\W_2^2( \cL(x_t,y_t),\cL(x_t, y_t')) \leq l_1 \E[|y_t -y_t'|^2]$, it holds that 
\begin{align*}
(-K +2 \k_1 - l_1 ) \Vert \ol{y}-\ol{y}' \Vert^2_{K} + \Vert \ol{z}-\ol{z}' \Vert^2_{K} \leq  l_1 \Vert y-y'\Vert^2_{K}.
\end{align*}
Since $-K +2 \k_1 - l_1 > l_1$, the mapping $(y_t) \mapsto (\ol{y}_t)$ is indeed a contraction.
\end{proof}

\begin{prop}\label{prop1}
Under Assumption~\ref{assume2}, for any $(\ol{y}_t) \in L^2_K(0,\infty,\R)$  there exists a unique solution $X$ to \eqref{eq:SDE}, and furthermore $X \in L^2_{K}(0,+\infty,\R)$. 
\end{prop}
\begin{proof}
The existence and uniqueness of solution to \eqref{eq:SDE} is standard (see e.g. \cite{MR3629171}). We only need to show that the unique solution $X$ belongs to the space $L^2_{K}(0,+\infty, \R)$.

Applying It\^{o}'s formula, it can be easily seen that 
\begin{align}\label{eq:prop2}
\E \left[e^{-K t} |X_t|^2 \right]=&\E[\xi^2]+ 2 \,\E \left[ \int_0^t e^{-K s} X_s \cdot B(s, X_s, \ol{y}_s, \cL(X_s,\ol{y}_s)) \, ds  \right] \notag \\
&-K \E\left[\int_0^s e^{-K s} |X_s|^2 \, ds \right]+\E \left[\int_0^t e^{-K s} \sigma^2 \, ds \right].
\end{align}
For the integrand of the second term on the right, we have that 
\begin{align*}
X_s \cdot B(s, X_s, \ol{y}_s,& \cL(X_s,\ol{y}_s)) = X_s \cdot \left (B(s, X_s, \ol{y}_s, \cL(X_s,\ol{y}_s))-B(s, 0, \ol{y}_s, \cL(X_s,\ol{y}_s))  \right) \\
& \quad \quad \quad \quad \quad \ \   +X_s \cdot B(s, 0, \ol{y}_s, \cL(X_s,\ol{y}_s)) \\
\leq & -\k_2 |X_s|^2+ |X_s| \cdot \left(|B(s,0,\ol{y}_s,\delta_{0} \otimes \cL(\ol{y}_s))| +l_2 \W_2(\delta_0 \otimes \cL(\ol{y}_s), \cL(X_s,\ol{y}_s) ) \right).
\end{align*}
With the fact that $\W_2(\delta_0 \otimes \cL(\ol{y}_s), \cL(X_s,\ol{y}_s)) \leq  \sqrt{\E[|X_s|^2 ] }$, one can easily derive that 
\begin{align*}
\E \left[X_s \cdot B(s, X_s, \ol{y}_s, \cL(X_s))\right] \leq ( -\k_2+l_2+ \e_2)  |X_s|^2+\frac{1}{4\e_2} \left( |B(s,0,\ol{y}_s,\delta_{0} \otimes \cL(\ol{y}_s))|^2 \right).
\end{align*}
Therefore from \eqref{eq:prop2}, we obtain that 
\begin{align*}
\E \left[e^{-K t} |X_t|^2 \right] \leq (-2\k_2+2l_2 -K+2 \e_2) \int_0^t e^{-K s} |X_s|^2 \, ds +C_{\e_2},
\end{align*}
where $C_{\e_2}$ is a constant depends on $K, \sigma, \E[\xi^2], l_2, \Vert B(\cdot, 0,0,\delta_0) \Vert^2_{K}, \Vert \ol{y} \Vert^2_{K}$. Due to Assumption~\ref{assume2} (iii), the coefficient before the integral on the right hand side is negative, and thus we conclude the $\Vert X \Vert^2_{K} < +\infty$. 
\end{proof}

\begin{thm}\label{thm2}
There exists a unique solution (X,Y,Z) to \eqref{eq:gMFB} in $L^2_{K}(0,\infty, \R^3)$.
\end{thm}
\begin{proof}
For any $x \in L^2_{K}(0,\infty, \R)$ such that $x_0=\xi$, define $\Psi(x):= (\ol{y},\ol{z}) \in L^2_{K}(0,\infty, \R^2)$ to be the unique solution to \eqref{eq:BSDE}, and for any $(\ol{y},\ol{z}) \in L^2_{K}(0,\infty, \R^2)$, define $\Phi(\ol{y},\ol{z}):=X \in L^2_{K}(0,\infty, \R)$ to be the unique solution to \eqref{eq:SDE}. We prove that the composition $\Phi \circ \Psi: L^2_{K}(0,\infty,\R) \to L^2_{K}(0,\infty, \R)$ is a contraction, and hence the fixed point of $\Phi \circ \Psi$ provides the unique solution to \eqref{eq:gMFB}. Take $x, x' \in L^2_{K}(0,\infty,\R)$ such that $x_0=x'_0=\xi$, $(\ol{y},\ol{z})=\Psi(x)$, $(\ol{y}',\ol{z}')=\Psi(x')$, and $X=\Phi(\ol{y},\ol{z})$, $X'=\Phi(\ol{y}',\ol{z}')$.

From It\^{o}'s formula, one can easily deduce that 
\begin{align}\label{thm2:1}
-K e^{-K t} |\ol{y}_t-\ol{y}_t'|^2 \,dt +e^{-K t} |\ol{z}_t-\ol{z}_t'|^2 \, dt= d e^{-K t} |\ol{y}_t-\ol{y}_t'|^2-2 e^{-K t} (\ol{y}_t-\ol{y}_t') \, d (\ol{y}_t-\ol{y}_t').
\end{align}
Since $\ol{y}, \ol{y}' \in L^2_{K}(0,\infty, \R)$, there exists a sequence of $T_i \to \infty$ such that $\E\left[e^{-K T_i} |\ol{y}_{T_i}-\ol{y}_{T_i}'|^2 \right] \to 0$.
Integrating \eqref{thm2:1} over interval $[0,T_i]$, taking expectation, and letting $T_i \to \infty$, we obtain that 
\begin{align}\label{thm2:2}
\E& \left[ \int_0^{\infty} -K e^{-K t} |\ol{y}_t-\ol{y}_t'|^2  + e^{-K t} | \ol{z}_t-\ol{z}_t'|^2 \,dt  \right]= -\E\left[|\ol{y}_0-\ol{y}_0'|^2 \right] \notag \\
& \quad \quad \quad \quad +\E \left[\int_0^{\infty} 2 e^{-K t} (\ol{y}_t-\ol{y}_t') \left( F(t,x_t,\ol{y}_t, \cL(x_t,\ol{y}_t))-F(t,x_t',\ol{y}_t', \cL(x_t',\ol{y}_t'))\right)  dt \right].
\end{align}
For the second term on the right hand side, we have that 
\begin{align*}
&2 e^{-K t} (\ol{y}_t - \ol{y}_t')\left( F(t,x_t,\ol{y}_t, \cL(x_t,\ol{y}_t) )-F(t,x_t',\ol{y}_t', \cL(x_t',\ol{y}_t') ) \right) \\
&   \leq -2  \k_1   e^{-K t} |\ol{y}_t - \ol{y}_t'|^2 + 2 l_1 e^{-K t} |\ol{y}_t-\ol{y}_t'| \left( |x_t-x_t'|+\W_2(\cL(x_t,\ol{y}_t), \cL(x_t',\ol{y}_t')) \right) \\
&   \leq -2  \k_1   e^{-K t} |\ol{y}_t - \ol{y}_t'|^2 + 2 l_1 e^{-K t} |\ol{y}_t-\ol{y}_t'| \left( |x_t-x_t'|+\sqrt{\E[|x_t -x_t'|^2]}+\sqrt{\E[|\ol{y}_t -\ol{y}_t'|^2]} \right).
\end{align*}
Therefore it holds that 
\begin{align}\label{thm2:3}
(-K +2 \k_1 -2l_1-2 l_1 \e_1) \Vert \ol{y}-\ol{y}' \Vert^2_{K} + \Vert \ol{z}-\ol{z}' \Vert^2_{K} \leq \frac{2l_1}{\e_1} \Vert x-x'\Vert^2_{K}.
\end{align}

Applying It\^{o}'s formula to $d e^{-K t} |X_t-X_t'|^2$, similarly we obtain that 
\begin{align*}
 \E &\left[\int_0^{\infty} K e^{-K t} |X_t -X_t'|^2 \, dt  \right]\\
&=\E \left[ \int_0^{\infty} 2 e^{-K t} (X_t-X_t' )\left(B(t,X_t, \ol{y}_t, \cL(X_t,\ol{y}_t))-B(t,X_t', \ol{y}_t', \cL(X_t',\ol{y}_t')) \right) dt \right]
\end{align*}
Note that 
\begin{align*}
&\E \left[ 2 e^{-K t} (X_t-X_t' )\left(B(t,X_t, \ol{y}_t, \cL(X_t,\ol{y}_t))-B(t,X_t', \ol{y}_t', \cL(X_t',\ol{y}_t')) \right) \right] \\
& \quad \quad  \leq \E \left[ -2 \k_2 e^{-K t} |X_t-X_t'|^2+2 l_2 e^{-K t} |X_t-X_t'| \left(| \ol{y}_t-\ol{y}_t'|+ \W_2 ( \cL(X_t,\ol{y}_t), \cL(X_t',\ol{y}_t')) \right)\right]  \\
& \quad \quad \leq (-2 \k_2+2 l_2+ 2 l_2 \e_2)  \E \left[ e^{-K t} |X_t-X_t'|^2 \right]+ \frac{2 l_2} {\e_2} \E\left[e^{-K t} |\ol{y}_t-\ol{y}_t'|^2 \right],
\end{align*}
and therefore 
\begin{align}\label{thm2:4}
(K+2 \k_2- 2 l_2 -2 l_2 \e_2) \Vert X-X' \Vert^2_{K} \leq \frac{2 l_2}{\e_2} \Vert \ol{y}-\ol{y}' \Vert^2_{K}.
\end{align}

According to Assumption~\ref{assume2} (iii), \eqref{thm2:3}, \eqref{thm2:4}, it can be easily seen that 
\begin{align*}
\Vert X-X' \Vert^2_{K} \leq& \frac{2 l_2}{\e_2(K+2 \k_2- 2 l_2 -2 l_2 \e_2)  } \Vert \ol{y}-\ol{y}' \Vert^2_{K} \\
\leq & \frac{ 4l_1 l_2 }{\e_1 \e_2 (-K +2 \k_1 -2l_1-2l_1 \e_1)(K+2 \k_2- 2 l_2 -2 l_2 \e_2) } \Vert x-x' \Vert^2_{K} <\Vert x-x' \Vert^2_{K},
\end{align*}
and therefore $\Phi \circ \Psi$ is a contraction.

\end{proof}

\section{Infinite horizon mean field game and mean field type control}\label{sec3}
In this section, we apply our main results to solve the infinite horizon mean field type control problem and  the infinite horizon mean field game. First in Subsection~\ref{subsec3.1}, we derive the corresponding McKean-Vlasov FBSDEs \eqref{eq:mfc}  and \eqref{eq:mfg} by Pontryagin's maximum principle, and solve the problems given solutions to \eqref{eq:mfc} and \eqref{eq:mfg}. Then in Subsection~\ref{subsec3.2}, we provide sufficient conditions for the existence of solutions to \eqref{eq:mfc} and \eqref{eq:mfg}.  Let $r>0$ be a discount factor and $A \subset \R$ be a convex control space. Suppose $b,f:  \R_+  \times \R \times \mathcal{P}_2(\R) \times A \to \R$ are two measurable functions. We work under the following assumption.
\begin{assume}\label{assume3}
(i) $b(t,x,\mu,a)$ is Lipschitz in $(x,\mu,a)$, and $f(t,x,\mu,a)$ is of at most quadratic growth in $(x,\mu,a)$. There exists a positive constant $l$ such that for any $\mu, \mu' \in \cP_2(\R)$, $t \in \R_+$, $x \in R$, $a \in A$,
\begin{align*}
|b(t,x,\mu,a)-b(t,x,\mu',a)| \leq l \W_2(\mu, \mu').
\end{align*}
(ii) $||b(\cdot,0,\delta_0,a)||^2_{r} < +\infty$, $\int_0^{\infty} e^{-rt} |f(t,0,\delta_0,a)| \, dt < +\infty$ for some (and thus any) $a \in A$. \\
(iii) There exists a constant $\k > l-\frac{r}{2}$ such that for any $t >0$, $a \in A$, $\mu \in \cP_2(\R)$, $x,x' \in \R$, it holds that
\begin{align*}
(x-x')(b(t,x,\mu, a)-b(t,x',\mu,a)) \leq -\k (x-x')^2.
\end{align*}
\end{assume}

\subsection{Pontryagin's maximum principle}\label{subsec3.1}
 
Define $\mathcal{A}:=L^2_{r}(0,\infty,A)$ to be the space of all admissible controls. For any control $\a \in \mathcal{A}$, let $(X_t)$ be a strong solution to the following controlled McKean-Vlasov SDE
\begin{align*}
\begin{cases}
dX_t =b(t,X_t, \cL(X_t), \a_t) \, dt + \sigma \, dW_t, \\
X_0=\xi. 
\end{cases}
\end{align*}
 As in the proof of Proposition~\ref{prop1}, it can be easily shown that under Assumption~\ref{assume3}, we have that $(X_t) \in L^2_{r}(0,\infty,\R)$. The cost functional takes the form 
\begin{align*}
J(\a):= \E \left[ \int_0^{\infty} e^{-rt} f(t,X_t,\cL(X_t),\a_t) \, dt \right],
\end{align*}
which is finite for any $\a \in \mathcal{A}$ due to Assumption~\ref{assume3}. We want to solve the minimization problem 
\begin{align}\label{controlproblem}
\inf_{\a \in \mathcal{A}} J(\a).
\end{align}
 Let us formally derive the maximum principle of the mean field  type control problem. Suppose $\a$ is an optimal control. Choose another admissible control $\beta$, define $\a^{\e}:=\a+\e\beta$, and denote by $X^{\e}$ the state trajectory corresponding to the control $\a^{\e}$. 
Let $$V_t= \lim\limits_{\e \to 0} \frac{X^{\e}_t-X_t}{\e}$$ be the variation process. Introduce the short-hand notation 
 \begin{align*}
 \theta_t:=(X_t, \cL(X_t), \a_t), \quad \theta^{\e}_t= (X^{\e}_t, \cL(X^{\e}_t), \a^{\e}_t).
 \end{align*}
Then it can be shown that $V$ satisfies
\begin{align*}
dV_t&= \left( \pa_x b(t,\theta_t) \cdot V_t + \wt{\E}\left[ \pa_{\mu} b(t,\theta_t )(\wt{X}_t) \cdot \wt{V}_t \right]+\pa_a b(t,\theta_t)\cdot \beta_t  \right) dt, \\
V_0&=0,
\end{align*}
where $(\wt{X}, \wt{V})$ is an independent copy of $(X,V)$ defined on $(\wt{\Omega}, \wt{\mathcal{F}},\wt{ \P})$ and $\wt{\E}\left[ \pa_{\mu} b(t,\theta_t )(\wt{X}_t) \cdot \wt{V}_t \right]$ is the derivative on the probability measure space when the state variable and the control are fixed, i.e.,
\begin{align}\label{eq:mudiff}
 \left.\wt{\E}\left[ \pa_{\mu} b(t,x,\cL(X_t), a )(\wt{X}_t) \cdot \wt{V}_t \right]\right|_{x=X_t,a=\a_t}.
\end{align}

To make \eqref{eq:mudiff} clear, in the following remark we briefly introduce how to differentiate functions of probability measures. We refer readers to \cite[Chapter 5]{MR3752669} for a nice survey on this topic. 
\begin{remark}
Let $\ol{\Omega}$ be a polish space and $(\ol{\P},\ol{\mathcal{F}})$ be an atomless probability measure over $\ol{\Omega}$. For any function $u: \cP_2(\R) \to \R$, we define its lift to the Hilbert space $L^2(\ol{\Omega}, \ol{\mathcal{F}},\ol{\P}; \R)$ by $\ol{u}(X):=u(\cL(X))$. Then $u$ is said to differentiable at $\mu_0=\cL(X_0)$ if $\ol{u}$ is Fr\'{e}chet differentiable at $X_0$. By identifying $L^2(\ol{\Omega}, \ol{\mathcal{F}},\ol{\P}; \R)$ with its dual, the Fr\'{e}chet derivative of $\ol{u}$ at $X_0$, denoted by $D \ol{u}(X_0)$, is an element in  $L^2(\ol{\Omega}, \ol{\mathcal{F}},\ol{\P}; \R)$. It can be shown that there exists a measurable function $\pa_{\mu} u(\mu_0):\R \to \R$ such that $\pa_{\mu} u(\mu_0)(X_0)=D \ol{u}(X_0)$ $\P$-a.s. Therefore we define the derivative of $u$ at $\mu_0$ as the measurable function $\pa_{\mu} u (\mu_0)$, which satisfies 
\begin{align*}
u(\mu)=u(\mu_0)+\E\left[\pa_{\mu} u(\mu_0)(X_0)\cdot (X-X_0) \right]+o(||X-X_0||_2),
\end{align*}
where $\cL(X)=\mu, \cL(X_0)=\mu_0$. 
\end{remark}

The function $\a \to J(\a)$ is G\^{a}teaux differentiable in the direction $\beta$ and its derivative is given by 
\begin{align*}
\left. \frac{d}{d \e} J(\a+\e \beta) \right|_{\e=0} = \E \left[\int_0^{\infty }e^{-rt} \left( \pa_x f(t,\theta_t) \cdot V_t + \wt{\E} \left[\pa_{\mu} f(t,\theta_t)(\wt{X}_t)\cdot \wt{V}_t \right]+\pa_a f(t,\theta_t) \cdot \beta_t  \right) dt \right] .
\end{align*}
Define the generalized Hamiltonian 
\begin{align}\label{eq:H}
\cH(t,x,{\mu},a,y):= b(t,x,{\mu},a) \cdot y+f(t,x,{\mu},a)-r x  y.
\end{align}
We consider the following infinite horizon BSDE 
\begin{align}\label{eq:BSDEa}
dY_t =- \left( \pa_x \cH(t, \Theta_t)+ \wt{\E} \left[ \pa_{\mu} \cH(t,\wt{\Theta}_t)(X_t)\right] \right) dt +Z_t \, dW_t,
\end{align}
where $\Theta_t:=(\theta_t,Y_t)=(X_t, \cL(X_t), \a_t, Y_t)$ and $(\wt{\Theta},\wt{\Omega}, \wt{\mathcal{F}},\wt{ \P})$ is an independent copy of $(\Theta, \Omega, \mathcal{F}, \P)$. 

Applying It\^{o}'s formula to the process $(e^{-rt}V_tY_t)$, it can be easily seen that
\begin{align*}
\left. \frac{d}{d \e} J(\a+\e \beta) \right|_{\e=0} = \E \left[ \int_0^{\infty} e^{-rt} \pa_a \cH(t, \Theta_t) \cdot \beta_t \, dt \right].
\end{align*}
Thus when $\a$ is an optimal admissible control with the associated stochastic processes $(X_t,Y_t,Z_t)$, it holds that
\begin{align*}
 \cH(t, X_t, \cL(X_t), \a_t, Y_t ) =\min_{a \in A}  \cH(t, X_t, \cL(X_t), a, Y_t ) \quad \quad \quad \text{Leb} \otimes \P \ \ a.e. 
\end{align*}

For any $x, y \in \R$, $m \in \cP_2(\R^2)$ with first marginal $\mu \in \cP_2(\R)$,  define 
\begin{align}\label{eq:amin}
\hat{\a}_t(x,y,{\mu})= \argmin_{a \in A} \cH(t,x,{\mu},a,y),
\end{align}
and 
\begin{align*}
& B_c(t, x,y,m):= b(t,x,\mu,\hat{\a}_t(x,y, \mu)), \\
& F_c(t,x,y,m):= \pa_x \cH(t, x, \mu, \hat{\a}_t(x,y, \mu), y)+ \int_{x',y'} \pa_{\mu}\cH(t,x',\mu,\hat{\a}_t(x',y',\mu),y')(x) \, dm(x',y').
\end{align*}
The above discussion connects the infinite horizon mean field control problem to the McKean-Vlasov FBSDE 
\begin{align}\label{eq:mfc}
\begin{cases}
dX_t = B_c(t,X_t, Y_t , \cL(X_t, Y_t)) \, dt + \sigma \, dW_t, \\
dY_t=-F_c(t, X_t, Y_t , \cL(X_t, Y_t))\, dt +Z_t \, dW_t, \\
X_0=\xi. 
\end{cases}
\end{align}
\begin{prop}\label{prop:control}
Let $(b,f)$ be differentiable in $(x,\mu,a)$, Assumption~\ref{assume3} hold and $\cH$ be convex in $(x,\mu,a)$. Suppose $||\hat{\a}_{\cdot}(0,0,\delta_0)||^2_{r} < +\infty$, $\hat{\a}_t$ is Lipschitz and $(B_c, F_c)$ satisfies either Assumption~\ref{assume1} or \ref{assume2} with $K=r$. Then we have that $J(\hat{\a})= \min_{\a} J(\a)$. 
\end{prop}
The convexity of $\cH$ is described by 
\begin{align*}
\cH(t,x',\mu',a',y)  \geq & \cH(t,x,\mu,a,y) + \pa_x \cH(t,x,\mu,a,y) \cdot (x'-x) \\
&+ \pa_a \cH(t,x,\mu,a,y) \cdot (a'-a) +\wt{\E}\left[\pa_{\mu} \cH(t,x,\mu,a,y) \cdot (\wt{X}'-\wt{X}) \right],
\end{align*}
where $x',x \in \R$, $a', a \in A$, $\mu',\mu \in \cP_2(\R)$, $\wt{X}', \wt{X}$ are defined on $(\wt{\Omega}, \wt{\mathcal{F}}, \wt{\mathbb{P}} )$, and have distributions $\mu',\mu$ respectively. 

Here we adopt the definition of $L$-convex functionals on $\cP_2(\R)$ in \cite[Section 5.5]{MR3752669}. It is equivalent to the well-known displacement convexity when functionals are continuously differentiable on $\cP_2(\R)$, see \cite[Proposition 5.79]{MR3752669}.
\begin{proof}
Due to Theorem~\ref{thm1}, \ref{thm2}, there exists a unique solution $(X,Y,Z)$ to \eqref{eq:mfc}. Let us denote $\theta^\wedge_t:=(X_t, \cL(X_t), \hat{\a}_t(X_t, Y_t, \cL(X_t)))$ and $\Theta^\wedge_t:=(\theta^\wedge_t, Y_t)$. 
For an arbitrary admissible control $\a'$ and its associated process $X'$, we have that 
\begin{align}\label{thm3:eq1}
J(\hat{\a})-J(\a')=& \E \left[\int_0^{\infty}e^{-rt} \left( \cH(t,X_t,\cL(X_t),\hat{\a}_t, Y_t) - \cH(t, X'_t,\cL(X'_t),\a'_t, Y_t)\right) \, dt  \right] \notag \\
&- \E \left[ \int_0^{\infty} e^{-rt} \left(b(t,X_t, \cL(X_t),\hat{X}_t)-b(t,X'_t, \cL(X'_t),\a'_t) \right) \cdot Y_t \, dt \right] \notag \\
&+r \E \left[\int_0^{\infty} e^{-rt} (X_t-X'_t) \cdot Y_t \, dt  \right].
\end{align}
It can be easily seen that there exists a sequence of $T_i \to \infty$ such that $\E\left[e^{-rT_i} (X_{T_i}-X_{T'_t}) \cdot Y_{T_i} \right] \to 0$. Applying It\^{o}'s formula to $e^{-rT_i} (X_{T_i}-X_{T'_t}) \cdot Y_{T_i}$ and letting $T_i \to \infty$, we obtain that 
\begin{align}\label{thm3:eq2}
&\E \left[ \int_0^{\infty} e^{-rt} (X_t-X'_t) \left(\pa_x \cH(t, \Theta^\wedge_t)+\wt{\E}\left[\pa_{\mu} \cH(\wt{\Theta}^\wedge_t)(X_t) \right]\right) \,dt \right] \notag \\
&= \E \left[ \int_0^{\infty} e^{-rt} \left(-r(X_t-X'_t)+b(t,X_t, \cL(X_t),\hat{X}_t)-b(t,X'_t, \cL(X'_t),\a'_t) \right) \cdot Y_t \, dt \right] .
\end{align}
According to the convexity of $\cH$ and the fact that 
$\hat{\a}_t= \argmin_{a \in A} \cH(t,X_t,\cL(X_t),a,Y_t)$, it holds that 
\begin{align}\label{thm3:eq3}
&   \cH(t, X'_t,\cL(X'_t),\a'_t, Y_t)-\cH(t,X_t,\cL(X_t),\hat{\a}_t, Y_t) \notag \\
&\geq (X'_t-X_t) \cdot \pa_x \cH(t,\Theta^\wedge_t) + \wt{\E} \left[\pa_{\mu} \cH(t, \Theta^\wedge_t)(\wt{X}_t) \cdot (\wt{X}'_t-\wt{X}_t )  \right]+(\a'_t-\hat{\a}_t) \cdot \pa_a \cH(t, \Theta^\wedge_t) \notag \\
& \geq (X'_t-X_t) \cdot \pa_x \cH(t,\Theta^\wedge_t) + \wt{\E} \left[\pa_{\mu} \cH(t, \Theta^\wedge_t)(\wt{X}_t)\cdot (\wt{X}'_t-\wt{X}_t )  \right].
\end{align}
By Fubini's theorem, we have that 
\begin{align*}
\E \left[(X'_t-X_t) \cdot \wt{\E}\left[\pa_{\mu} \cH(\wt{\Theta}^\wedge_t)(X_t) \right] \right]= \E \wt{\E} \left[\pa_{\mu} \cH(t, \Theta^\wedge_t)(\wt{X}_t)\cdot (\wt{X}'_t-\wt{X}_t ) \right].
\end{align*}
In conjunction with \eqref{thm3:eq1}, \eqref{thm3:eq2}, \eqref{thm3:eq3}, we conclude that 
\begin{align*}
&J(\hat{\a})-J(\a') \leq  \E \left[\int_0^{\infty}e^{-rt} \left( \cH(t,X_t,\cL(X_t),\hat{\a}_t, Y_t) - \cH(t, X'_t,\cL(X'_t),\a'_t, Y_t)\right) \, dt  \right] \\
&- \E \left[ \int_0^{\infty} e^{-rt} (X_t-X'_t) \left(\pa_x \cH(t, \Theta^\wedge_t)+\wt{\E}\left[\pa_{\mu} \cH(\wt{\Theta}^\wedge_t)(X_t) \right]\right) \,dt \right] \\
 & \leq \E \left[\int_0^{\infty}e^{-rt} \left( \cH(t,X_t,\cL(X_t),\hat{\a}_t, Y_t) - \cH(t, X'_t,\cL(X'_t),\a'_t, Y_t)\right) \, dt  \right] \\
&- \E \left[ \int_0^{\infty} e^{-rt} \left( (X_t-X'_t) \cdot \pa_x \cH(t,\Theta^\wedge_t) + \wt{\E} \left[\pa_{\mu} \cH(t, \Theta^\wedge_t)(\wt{X}_t)\cdot (\wt{X}_t-\wt{X}'_t )  \right] \right) \, dt \right] 
\leq 0. 
\end{align*}
\end{proof}

Now we introduce an infinite horizon mean field game with discounted cost. Suppose there are $N$ players, and each player $i$ has state variable $X_t^i$ at time $t$. Denote the empirical distribution of $N$ players by $\overline{\mu}_t := \frac{1}{N} \sum_{i=1}^N \delta_{X_t^i}$.  Given admissible controls $\a^1, \dotso, \a^N \in \mathcal{A}$ and $N$ independent Brownian motions $W^1 ,\dotso, W^N$, the players have dynamics
\begin{align}
dX_t^i= b(t,X_t^i, \overline{\mu}_t, \alpha^i_t) \, dt+ \sigma  \, dW_t^i, \quad  i=1, \dotso ,N. 
\end{align}
The cost functional for player $i$ is given by 
\begin{align}
J^i(\a^1, \dotso, \a^N):= \E \left[ \int_0^{\infty} e^{-rt} f(t,X_t^i, \overline{\mu}_t, \a^i_t) \, dt \right],
\end{align}
where $r>0$ is the discount factor and $f:\R_+ \times \R \times \cP_2(\R) \times A \to \R$ is the running cost. We want to study the Nash equilibrium as $N \to \infty$. 

Suppose $\ol{\mu}_t$ converges to a measure flow $\mu_t$ in equilibrium as $N \to \infty$. Then a representative player wants to minimize
\begin{align*}
J^{\mu}(\a):=\E \left[  \int_0^{\infty} e^{-rt} f(t,X_t, \mu_t, \a_t) \, dt\right],
\end{align*}
under the constraint 
\begin{align*}
dX_t =b (t,X_t,\mu_t,\a_t) \, dt + \sigma \, dW_t.
\end{align*}
As the variational argument for the mean field type control problem, the optimal strategy of the representative should be given by $\hat{\a}(t, X_t, Y_t, \mu_t)$ where $(X,Y,Z)$ is the solution to 
\begin{align}\label{eq:mfg0}
\begin{cases}
&dX_t =b \left(t,X_t,\mu_t,\hat{\a}_t(X_t,Y_t,\mu_t)\right) \, dt + \sigma \, dW_t, \\
&dY_t =-\pa_x \cH \left(t,X_t,\mu_t,\hat{\a}_t(X_t,Y_t,\mu_t),Y_t\right) \, dt +Z_t \, dW_t, \quad \forall {t \geq 0}, \\
&X_0=\xi.
\end{cases}
\end{align}

 For any $m \in \cP_2(\R^2)$ with first marginal $\mu \in \cP_2(\R)$, define
\begin{align*}
&B_g(t,x,y,m):= b(t,x,\mu,\hat{\a}_t(x,y,\mu)), \\
&F_g(t,x,y,m):=-\pa_x \cH(t,x,\mu,\hat{\a}(x,y,\mu),y).
\end{align*}
It also required that the law of $X_t$ coincides with $\mu_t$. Thus plugging $\mu_t= \cL(X_t)$ in \eqref{eq:mfg0}, we obtain the FBSDE of mean field game
\begin{align}\label{eq:mfg}
\begin{cases}
&dX_t =B_g(t,X_t,Y_t, \cL(X_t,Y_t)) \, dt + \sigma \, dW_t, \\
&dY_t =-F_g(t,X_t,Y_t,\cL(X_t,Y_t))\, dt +Z_t \, dW_t, \quad \forall {t \geq 0}. \\
&X_0=\xi.
\end{cases}
\end{align}

\begin{prop}\label{prop:game}
Let $(b,f)$ be differentiable in $(x,a)$, Assumption~\ref{assume3} hold and $\cH$ be convex in $(x,a)$. Suppose $||\hat{\a}_{\cdot}(0,0,\delta_0)||^2_{r} < +\infty$, $\hat{\a}_t$ is Lipschitz and $(B_g,F_g)$ satisfies either Assumption~\ref{assume1} or \ref{assume2} with $K=r$. Then there exists a unique solution $(X,Y,Z) \in L^2_{r}(0,\infty, \R^3)$ to \eqref{eq:mfg} which provides an equilibrium to the infinite horizon mean field game, i.e., 
\begin{align*}
J^{\cL(X)}(\hat{\a}) \leq J^{\cL(X)}(\a),  \quad \forall \a \in \mathcal{A}. 
\end{align*}
\end{prop}
\begin{proof}
Given the existence of solutions to \eqref{eq:mfg}, the proof is standard, see e.g. \cite[Theorem 3.17]{MR3752669}.
\end{proof}

\begin{remark}
In the mean field game, since there are large number of players, any change of a representative player doesn't impact the measure flow $(\mu_t)$. Therefore $(\mu_t)$ is fixed in the derivation of \eqref{eq:mfg0}. That's the main difference from mean field control problem, where the law $\cL(X_t)$ changes as the control changes. For more detailed discussions, see e.g. \cite{MR3045029}. 
\end{remark}

\subsection{Solvability of Mean field type control and Mean field game FBSDEs}\label{subsec3.2}

In this subsection, we find sufficient conditions on the given data for the existence and uniqueness of solutions to \eqref{eq:mfc} and \eqref{eq:mfg}. For the mean field type control problem, we assume that $b(t,x,\mu,a)=b_0(t)+\ol{b}_1(t)\ol{\mu}+b_1(t)x+b_2(t)a$, where $b_0(t), \ol{b}_t(t),b_1(t), b_2(t)$ are deterministic functions. For the mean field game problem, we assume that $b(t,x,\mu,a)=b_0(t,\mu)+b_1(t)x+b_2(t)a$, where by abuse of notation $b_0(t,\cdot)$ is a measurable function of $\mu \in \cP_2(\R)$ for any $t \in \R_+$.  Let us compute $(B_c,F_c)$,
\begin{align}\label{eq:Fc}
B_c(t,x,y,m)=&b_0(t)+\ol{b}_1(t)\ol{\mu}+b_1(t)x+b_2(t) \hat{\a}_t(x,y, \mu), \notag \\
F_c(t,x,y,m)=&b_1(t)y +\pa_x f(t,x,\mu, \hat{\a}_t(x,y,\mu))-ry  \notag \\
&+ \ol{b}_1(t) \ol{\nu}+ \int_{x',y'} \pa_{\mu} f(t,x',\mu,\hat{\a}_t(x',y',\mu))(x) \, dm (x',y'),
\end{align}
where $\mu$ is the first marginal of $m$.
\begin{definition} A continuously differentiable function $\rho:\R \to \R$ is said to be $\eta$-convex for some $\eta>0$ if
\begin{align*}
\rho(z')-\rho(z) -(z'-z)\cdot \pa_z \rho(z) \geq \eta (z'-z)^2, \quad \forall z,z' \in \R.
\end{align*}
\end{definition}
First, we show the Lipschitz and convex property of the minimizer $\hat{\a}_t$ \eqref{eq:amin}. 
\begin{lemma}\label{lem:propertya}
 Suppose $b(t,x,\mu,a)=b_0(t,\mu)+b_1(t)x + b_2(t)a$, $f$ is once continuously differentiable in $(x,a)$, $\eta$-convex in $a$, and  $\pa_a f$ is $l$-Lipschitz in $(\mu,x)$. Then it holds that 
\begin{align}\label{lem4:1}
 |\hat{\a}_t(x,y,\mu)-\hat{\a}_t(x',y',\mu')| \leq  \frac{l}{2\eta} |x'-x|+\frac{|b_2(t)|}{2\eta} |y'-y|+\frac{l}{2 \eta} \W_2(\mu,\mu'),
\end{align}
and for any $(t,x,y,\mu) \in \R_+ \times \R^2 \times \cP_2(\R)$, 
\begin{align}\label{lem4:2}
|\hat{\a}_t(x,y,\mu)| \leq \eta^{-1}(|\pa_a f(t,x,\mu,a_0)|+|b_2(t) y|)+|a_0|.
\end{align}
Furthermore, if $A=\R$ and $\pa_a f$ is $\zeta$-Lipchitz in $a$, it follows that 
\begin{align}\label{lem4:4}
b_2(t)(y'-y) \cdot \left(\hat{\a}_t(x,y',\mu)-\hat{\a}_t(x,y,\mu) \right) \leq -\frac{2b_2(t)^2 \eta}{ \zeta^2 } (y'-y)^2.
\end{align}
\end{lemma}
\begin{proof}
The proofs of \eqref{lem4:1} and \eqref{lem4:2} are from \cite[Lemma 3.3, Lemma 6.18]{MR3752669}. Denote $\hat{\a}_t=\hat{\a}_t(x,y,\mu)$ and $\hat{\a}_t'=\hat{\a}_t(x,y',\mu)$. In the case that $A=\R$, it is clear that $\pa_a \cH(t,x,\mu,\hat{\a}_t,y)= \pa_a \cH(t,x,\mu,\hat{\a}_t',y')=0$, and thus
\begin{align}\label{lem4:5}
b_2(t) (y'-y)+ \left( \pa_a f(t,x,\mu,\hat{\a}_t')- \pa_a f(t,x,\mu,\hat{\a}_t)\right)=0.
\end{align}
Since $f$ is $\eta$-convex in $a$ and $\pa_a f$ in $\zeta$-Lipschitz in $a$, we obtain that
\begin{align*}
\eta (\hat{\a}_t'-\hat{\a}_t)^2 \leq & f(t,x,\mu,\hat{\a}_t')-f(t,x,\mu,\hat{\a}_t)- (\hat{\a}_t'-\hat{\a}_t) \cdot \pa_a f(t,x,\mu,\hat{\a}_t) \\
= &   (\hat{\a}_t'-\hat{\a}_t) \cdot \int_0^1 \pa_a f(t,x, \mu, \hat{\a}_t+s(\hat{\a}_t'-\hat{\a}_t) )- \pa_a f(t,x,\mu,\hat{\a}_t) \, ds 
\leq  \frac{\xi}{2} (\hat{\a}_t'-\hat{\a}_t)^2. 
\end{align*}
For the same reason, we also have
\begin{align*}
 \eta (\hat{\a}_t'-\hat{\a}_t)^2 \leq f(t,x,\mu,\hat{\a}_t)-f(t,x,\mu,\hat{\a}_t')- (\hat{\a}_t-\hat{\a}_t') \cdot \pa_a f(t,x,\mu,\hat{\a}_t')  \leq \frac{\zeta}{2} (\hat{\a}_t'-\hat{\a}_t)^2,
\end{align*}
and therefore 
\begin{align*}
 2  \eta (\hat{\a}_t'-\hat{\a}_t)^2 \leq(\hat{\a}_t'-\hat{\a}_t) \cdot \left( \pa_a f(t,x,\mu,\hat{\a}_t')- \pa_a f(t,x,\mu,\hat{\a}_t)\right)  \leq  \zeta (\hat{\a}_t'-\hat{\a}_t)^2 .
\end{align*}
Multiplying \eqref{lem4:5} by $(\hat{\a}_t'-\hat{\a}_t)$ and using the above inequality, we get that
\begin{align*}
 2  \eta (\hat{\a}_t'-\hat{\a}_t)^2 \leq - b_2(t)(y'-y) \cdot (\hat{\a}_t'-\hat{\a}_t ) \leq  \zeta (\hat{\a}_t'-\hat{\a}_t)^2 ,
\end{align*}
and also 
\begin{align*}
 \frac{|b_2(t)|}{\zeta}|y'-y| \leq |\hat{\a}_t'-\hat{\a}_t| .
\end{align*}
Therefore we conclude that
\begin{align*}
b_2(t)(y'-y) \cdot (\hat{\a}_t'-\hat{\a}_t ) \leq - 2 \eta (\hat{\a}_t'-\hat{\a}_t)^2 \leq -\frac{2b_2(t)^2\eta}{ \zeta^2} (y'-y)^2. 
\end{align*}
\end{proof}

We show that the following function, as a part of $F_c$ \eqref{eq:Fc}, is Lipschitz 
\begin{align*}
\Psi:(t,x,m)  \mapsto  \ol{b}_1(t) \ol{\nu}+ \Phi(t,x,m),
\end{align*}
where 
\begin{align*}
\Phi(t,x,m)=\int_{x',y'} \pa_{\mu} f(t,x',\mu,\hat{\a}_t(x',y',\mu))(x) \, dm (x',y').
\end{align*}

\begin{lemma}\label{lem:lip}
Assume that  $f$ is once continuously differentiable in $(x,\mu,a)$, $\eta$-convex in $a$, $\pa_a f$ is $l$-Lipschitz in $(x,\mu)$, and $\pa_{\mu} f(t,x',\mu,a)(x)$ is $l$-Lipschitz in $(x',\mu,a,x)$. Then for any $x,\ol{x} \in \R$, $m,\ol{m} \in \cP_2(\R^2)$ it holds that 
\begin{align}
|\Psi(t,x,m)- \Psi(t,\ol{x},\ol{m})| \leq  \left(|\ol{b}_1(t)|+\frac{l(4\eta+2l+|b_2(t)|)}{2\eta} \right) \W_2(m,\ol{m})+l|x-\ol{x}|. \label{eq:lipstz}
\end{align}
\end{lemma}
\begin{proof}
Together with Lemma~\ref{lem:propertya}, we have the Lipschitz property 
\begin{align*}
& \left| \pa_{\mu} f(t,x',\mu,\hat{\a}_t(x',y',\mu))(\ol{x})-\pa_{\mu} f(t,\ol{x}',\mu,\hat{\a}_t(\ol{x}',\ol{y}',\mu))(\ol{x}) \right| \\
& \quad \quad \quad \leq l\left( |x'-\ol{x}'|+|\hat{\a}_t(x',y',\mu)-\hat{\a}_t(\ol{x}',\ol{y}',\mu)| \right) \\
& \quad \quad \quad \leq l\left(|x'-\ol{x}'|+\frac{l}{2\eta} |x'-\ol{x}'|+\frac{|b_2(t)|}{2\eta}|y'-\ol{y}'|\right).
\end{align*}
Therefore it holds that 
\begin{align*}
\left|\int_{x',y'} \pa_{\mu} f(t,x',\mu,\hat{\a}_t(x',y',\mu))(\ol{x}) \, d(m-\ol{m}) (x',y') \right| \leq \frac{l(2\eta+l+|b_2(t)|)}{2\eta} \W_2(m,\ol{m}),
\end{align*}
and hence
\begin{align*}
|\Psi(t,x,m)-& \Psi(t,\ol{x},\ol{m})| \leq |\ol{b}_1(t)| \W_1(\nu,\nu')+ |\Phi(t,x,m)-\Phi(t,\ol{x},m)|+|\Phi(t,\ol{x},m)-\Phi(t,\ol{x},\ol{m})| \notag \\
\leq & |\ol{b}_1(t)| \W_2(m,\ol{m})+l|x-\ol{x}|+\left|\int_{x',y'} \pa_{\mu} f(t,x',\mu,\hat{\a}_t(x',y',\mu))(\ol{x}) \, d(m-\ol{m}) (x',y') \right| \notag \\
&+ \int_{x',y'} \left| \pa_{\mu} f(t,x',\mu,\hat{\a}_t(x',y',\mu))(\ol{x})-\pa_{\mu} f(t,x',\ol{\mu},\hat{\a}_t(x',y', \ol{\mu}))(\ol{x})\right| \, d\ol{m} (x',y')  \notag \\
\leq & \left(|\ol{b}_1(t)|+\frac{l(4\eta+2l+|b_2(t)|)}{2\eta} \right) \W_2(m,\ol{m})+l|x-\ol{x}|. 
\end{align*}
\end{proof}

\begin{remark}
\cite[Lemma 5.41]{MR3752669} provides a sufficient condition for the Lipschitz property of $$(x',\mu,a,x) \mapsto \pa_{\mu}f(t,x',\mu,a)(x).$$
\end{remark}

\begin{thm}\label{thm3.1}
Let $b(t,x,\mu,a)=b_0(t)+\ol{b}_1(t)\ol{\mu}+b_1(t)x+b_2(t)a$. The conclusion of Proposition~\ref{prop:control} holds under either conditions $(i), (ii), (iii), (iv)$ or conditions $(i'), (ii'), (iii'), (iv')$ below, and thus $\hat{\a}$ solves the minimization problem~\eqref{controlproblem}.
\begin{enumerate}[(i)]
 \item  $b_1(t),b_2(t)$ are uniformly bounded, and there  exists a positive constant $l$ such that $|\ol{b}_1(t)| \leq l$ and $-\max_{t} b_1(t) \geq l- \frac{r}{2}$. $f$ is once continuously differentiable in $(x,\mu,a)$, of at most quadratic growth in $(x,\mu,a)$, and it holds that $b_0(\cdot), |f(\cdot,0,\delta_0,a)|^{1/2} \in L^2_{r}(0,\infty, \R)$ for some (any thus any) $a \in A$. 
\item There exist some positive constants $\eta,\iota$ such that the following convexity condition holds 
\begin{align*}
f(t,x',\mu',a')-f(t,x,\mu,a)& -\pa_{(x,a)}f(t,x,\mu,a) \cdot ( x'-x,a'-a) \\
& - \E\left[\pa_{\mu}f(t,x,\mu,a)(X) \cdot (X'-X) \right] \geq \iota (x'-x)^2+\eta(a'-a)^2,
\end{align*}
for any $t \in \R_+$ whenever $X',X$ have distributions $\mu',\mu$ respectively. 
\item $\pa_x f$ and $\pa_a f$ are $l$-Lipschitz in $(\mu,a)$ and $(x,\mu)$ respectively. $\pa_a f$ is $\zeta$-Lipschitz in $a$, and $\pa_{\mu} f(t,x',\mu,a)(x)$ is $l$-Lipschitz in $(x',\mu,a,x)$. 
\item $A=\R$, and it holds that 
\begin{align}\label{eq:condition1}
\inf_t \min &\left\{2 \iota-\frac{13l}{2}-\frac{5l^2+3|b_2(t)|l}{2\eta},  \frac{2b_2(t)^2\eta}{\zeta^2}-\frac{3l}{2}-\frac{l^2+2|b_2(t)|l}{2\eta} \right\} > \frac{r}{2}.
\end{align}
\end{enumerate}
\vspace{8pt}

\begin{enumerate}[(i')]
\item $b_1(t),b_2(t)$ are uniformly bounded, and there  exists a positive constant $l$ such that $|\ol{b}_1(t)| \leq l$. $f$ is once continuously differentiable in $(x,\mu,a)$, of at most quadratic growth in $(x,\mu,a)$, and it holds that $b(\cdot), |f(\cdot,0,\delta_0,a)|^{1/2} \in L^2_{r}(0,\infty, \R)$ for some (any thus any) $a \in A$. 
\item There exists a positive constant $\eta$ such that the following convexity condition holds 
\begin{align*}
f(t,x',\mu',a')-f(t,x,\mu,a)& -\pa_{(x,a)}f(t,x,\mu,a) \cdot ( x'-x,a'-a) \\
& - \E\left[\pa_{\mu}f(t,x,\mu,a)(X) \cdot (X'-X) \right] \geq \eta(a'-a)^2,
\end{align*}
for any $t \in \R_+$ whenever $X',X$ have distributions $\mu',\mu$ respectively. 
\item $\pa_x f$ and $\pa_a f$ are $l$-Lipschitz in $(x,\mu,a)$ and $(x,\mu)$ respectively. $\pa_{\mu} f(t,x',\mu,a)(x)$ is $l$-Lipschitz in $(x',\mu,a,x)$. 
\item It holds that 
\begin{align}\label{eq:condition2}
\max_t b_1(t) \leq - \max  \left\{9l-\frac{r}{2}+\max_t\frac{9l^2+4l|b_2(t)|}{2\eta},  3l-\frac{r}{2} +\max_t \frac{4|b_2(t)|l+3b_2(t)^2}{2\eta} \right\}.
\end{align}

\end{enumerate}
\end{thm}
\begin{proof}
Assume that conditions $(i), (ii), (iii), (iv)$ hold. It is clear that Assumption~\ref{assume3} is satisfied, and due to Lemma~\ref{lem:propertya} $\hat{\a}_t$ is Lipschitz and $\hat{\a}_{\cdot}(0,0,\delta_0) \in L^2_{r}(0,\infty, \R)$. According to condition $(ii)$, it can be easily seen that $\cH$ is convex in $(x,\mu,a)$. By Lemma~\ref{lem:lip} and explicit formulas of $(B_c,F_c)$ \eqref{eq:Fc}, Assumption~\ref{assume1} (i) can be easily verified. It remains to to check Assumption~\ref{assume1} (ii) with $K=r$. 

Take any square integrable random variables $X,Y,X',Y'$, and denote $\mu=\cL(X), \mu'=\cL(X'), m=\cL(X,Y), m'=\cL(X',Y')$. Define $\hat{X}=X-X', \hat{Y}=Y-Y'$ and $U=(X, Y , \cL(X,Y)), U'=(X', Y', \cL(X',Y')).$ Let us compute 
\begin{align}\label{thm3.1:eq1}
-r \hat{X} \hat{Y}&- \hat{X} (F_c(t,U)-F_c(t,U') )+\hat{Y} (B_c(t,U)-B_c(t,U') ) \notag \\
=&-\hat{X}\left(\pa_x f(t,X,\mu,\hat{\a}_t(X,Y,\mu))- \pa_x f(t,X',\mu',\hat{\a}_t(X',Y',\mu')) +\Psi(X,m)-\Psi(X',m')\right) \notag \\
&+ \hat{Y} \left(\ol{b}_1(t) \E[\hat{X}]+b_2(t) (\hat{\a}_t(X,Y,\mu)-\hat{\a}_t(X',Y',\mu')) \right).
\end{align}
Since $f$ is $\iota$-convex in $x$, we have that 
\begin{align}\label{thm3.1:eq2}
-\hat{X}&\left(\pa_x f(t,X,\mu,\hat{\a}_t(X,Y,\mu))- \pa_x f(t,X',\mu',\hat{\a}_t(X',Y',\mu')) \right) \notag \\
=&  -\hat{X}\left(\pa_x f(t,X,\mu,\hat{\a}_t(X,Y,\mu))- \pa_x f(t,X',\mu,\hat{\a}_t(X,Y,\mu)) \right) \notag \\
&-\hat{X}\left(\pa_x f(t,X',\mu,\hat{\a}_t(X,Y,\mu))- \pa_x f(t,X',\mu',\hat{\a}_t(X',Y',\mu')) \right) \notag \\
\leq &- 2 \iota \hat{X}^2+l |\hat{X}|\left(\W_2(\mu,\mu')+\frac{l}{2\eta} |\hat{X}|+\frac{|b_2(t)|}{2\eta} |\hat{Y}|+\frac{l}{2 \eta} \W_2(\mu,\mu')\right).
\end{align}
According to \eqref{lem4:4}, it follows that 
\begin{align}\label{thm3.1:eq3}
\hat{Y} b_2(t) &\left(\hat{\a}_t(X,Y,\mu)-\hat{\a}_t(X',Y',\mu')\right) \notag\\
=&\hat{Y} b_2(t) \left(\hat{\a}_t(X,Y,\mu)-\hat{\a}_t(X,Y',\mu)\right) +\hat{Y} b_2(t) \left(\hat{\a}_t(X,Y',\mu)-\hat{\a}_t(X',Y',\mu')\right) \notag\\
\leq & -\frac{2b_2(t)^2 \eta}{ \zeta^2 } \hat{Y}^2+|\hat{Y} b_2(t)|\left(\frac{l}{2\eta} |\hat{X}|+\frac{l}{2 \eta} \W_2(\mu,\mu') \right).
\end{align}
Using Lemma~\ref{lem:lip}, equations \eqref{thm3.1:eq1},\eqref{thm3.1:eq2},\eqref{thm3.1:eq3}, condition $(iv)$ and basic inequalities, Assumption~\ref{assume1} (ii) can be verified.

Assume that conditions $(i'), (ii'), (iii'), (iv')$ hold. We only check Assumption~\ref{assume2}, and the rest is very similar to the first part of proof. Recalling the formula \eqref{eq:Fc}, it can be easily verified that 
\begin{align*}
&(y-y')\left(F_c(t,x,y,m)-F_c(t,x,y',m) \right) \leq \left( b_1(t)-r+\frac{|b_2(t)|l}{2\eta} \right) (y-y')^2, \\
&(x-x')\left(B_c(t,x,y,m)-B_c(t,x',y,m) \right) \leq \left(b_1(t)+\frac{|b_2(t)|l}{2\eta} \right)(x-x')^2, \\
&|F_c(t,x,y,m)-F_c(t,x',y,m')|  \leq \left(3l+\frac{3l^2+|b_2(t)|l}{2 \eta} \right)\W_2(m,m')+\left( 2l+\frac{l^2}{2\eta}\right)|x-x'|,\\
&|B_c(t,x,y,m)-B_c(t,x,y',m')|  \leq  \left( l+\frac{|b_2(t)|l}{2\eta} \right) \W_2(m,m')+\frac{b_2(t)^2}{2\eta}|y-y'|.
\end{align*}
Therefore we define 
\begin{align*}
\k_1&=-\max_t \left( b_1(t)-r+\frac{|b_2(t)|l}{2\eta} \right), \\
\k_2&=-\max_t \left(b_1(t)+\frac{|b_2(t)|l}{2\eta} \right), \\
l_1&=\max_t \left(3l+\frac{3l^2+|b_2(t)|l}{2 \eta} \right), \\
l_2&=\max_t \left( l+\frac{|b_2(t)|l}{2\eta}+\frac{b_2(t)^2}{2 \eta}\right). 
\end{align*}
Due to condition $(iv')$, it can be easily verified that 
\begin{align*}
-2\k_1+6l_1 <-r<2\k_2-6l_2,
\end{align*}
and hence Assumption~\ref{assume2} (iii) is satisfied. 
\end{proof}

Now we provide sufficient conditions to solve \eqref{eq:mfg}. Assume that $b(t,x,\mu,a)=b_0(t,\mu)+b_1(t)x+b_2(t)a$. Then it is clear that 
\begin{align}\label{eq:coeffmfg}
B_g(t,x,y,\mu)=&b_0(t,\mu)+b_1(t)x+b_2(t) \hat{\a}_t(x,y, \mu), \notag\\
F_g(t,x,y,\mu)=&b_1(t)y +\pa_x f(t,x,\mu, \hat{\a}_t(x,y,\mu))-ry. 
\end{align}

\begin{thm}\label{thm3.2}
Let $b(t,x,\mu,a)=b_0(t,\mu)+b_1(t)x+b_2(t)a$. The conclusion of Proposition~\ref{prop:game} holds under either conditions $(i), (ii), (iii), (iv)$ or conditions $(i'), (ii'), (iii'), (iv')$ below, and thus $(\cL(X_t), \hat{\a}_t)$ solves the infinite horizon mean field game. 
\begin{enumerate}[(i)]
\item $b_1(t),b_2(t)$ are uniformly bounded, and $b_0(t,\mu)$ is $l$-Lipschitz in $\mu$, such that $-\max_{t} b_1(t) \geq l- \frac{r}{2}$. $f$  is once continuously differentiable in $(x,a)$, of at most quadratic growth in $(x,\mu,a)$, and it holds that $b(\cdot, \delta_0), |f(\cdot,0,\delta_0,a)|^{1/2} \in L^2_{r}(0,\infty, \R)$ for some (any thus any) $a \in A$. 
\item $f$ is $\iota$-convex in $x$ and $\eta$-convex in $a$.  
\item $\pa_x f$ and $\pa_a f$ are $l$-Lipschitz in $(\mu,a)$ and $(x,\mu)$ respectively. $\pa_a f$ is $\zeta$-Lipschitz in $a$.
\item $A=\R$ and it holds that
\begin{align}\label{eq:condition3}
\inf_t \min \left\{2\iota - \frac{3l}{2}-\frac{l^2}{\eta}-\frac{3|b_2(t)|l}{4 \eta}, \frac{2 b_2(t)^2 \eta}{ \zeta^2}-\frac{l}{2} -\frac{3|b_2(t)|l}{4\eta} \right\} \geq \frac{r}{2}.
\end{align}
\end{enumerate}
\vspace{8pt}
\begin{enumerate}[(i')]
\item $b_1(t),b_2(t)$ are uniformly bounded, and $b_0(t,\mu)$ is $l$-Lipschitz in $\mu$. $f$  is once continuously differentiable in $(x,a)$, of at most quadratic growth in $(x,\mu,a)$, and it holds that $b(\cdot, \delta_0)$, $|f(\cdot,0,\delta_0,a)|^{1/2} \in L^2_{r}(0,\infty, \R)$ for some (any thus any) $a \in A$. 
\item $f$ is $\eta$-convex in $a$, convex $x$.
\item $\pa_x f$ is $l$-Lipschitz in $(x,\mu,a)$, and $\pa_a f $ is $l$-Lipschitz in $(x,\mu)$. 
\item It holds that 
\begin{align}\label{eq:condition4}
\max_t b_1(t) \leq - \max  \left\{3l-\frac{r}{2}+\max_t\frac{3l^2+|b_2(t)|l}{2\eta},  3l-\frac{r}{2} +\max_t \frac{4|b_2(t)|l+3b_2(t)^2}{2\eta} \right\}.
\end{align}
\end{enumerate}
\end{thm}
\begin{proof}
The proof is almost the same as that of Theorem~\ref{thm3.1}.
\end{proof}

\begin{remark}
Using PDE tools, \cite{cardaliaguet2019,2021arXiv210109965C} studied the long time behavior of mean fields games in the special case when $b(t,x,\mu,a)=a$, $f(t,x,\mu,a)=L(x,a)+F(x,\mu)$.  Their main assumption, the uniform convexity of $y \mapsto - \inf_a \{ ay+L(x,a)\}$ fails whenever the control space $A$ is bounded. This is a case when Assumption~\ref{assume2} can prove to be less demanding since it holds for large enough $-b_1(t)$ no matter $A$ is bounded or not; see equation \eqref{eq:coeffmfg}.  

\cite{cardaliaguet2019,2021arXiv210109965C} proved that the vanishing discount limit for the infinite horizon problem is the solution to an ergodic mean field games \cite[Theorem 6.4]{2021arXiv210109965C}, and that the solution to the discounted mean field game converges to the unique stationary solution exponentially fast \cite[Theorem 3.7]{cardaliaguet2019}. It remains open to show the above convergence results for general models using FBSDE techniques, and we leave it for future research. 
\end{remark}
 
\section{Linear quadratic models}\label{sec4}

In this section, we apply Theorem~\ref{thm3.1}, \ref{thm3.2} to linear quadratic models. For any $\mu \in \cP_2(\R)$, define $\ol{\mu}:=\int x \, \mu(dx)$ as the mean of distribution $\mu$. Let us suppose $A= \R$, and
\begin{align*}
b(t,x,\mu,a):=& b_1(t)x+\overline{b}_1(t) \overline{\mu}+ b_2(t) a, \\
f(t,x,\mu,a):=& \frac{1}{2} \left(x^2q(t)+(x-\ol{\mu})^2\ol{q}(t)+ a^2 p(t)\right),
\end{align*}
where $b_1(t), \ol{b}_1(t), b_2(t),q(t), \ol{q}(t), p(t)$ are deterministic functions.

In this simple case, we can explicitly compute \eqref{eq:amin}
\begin{align*}
\hat{\a}_t(x,y,\mu)= - \frac{b_2(t)}{p(t)} y.
\end{align*}
Plugging in \eqref{eq:mfc} and \eqref{eq:mfg}, we obtain that 
\begin{align*}
B_c(t,x,y,m)&=B_g(t,x,y,m)=b_1(t)x-\frac{b_2(t)^2}{p(t)}y + \ol{b}_1(t) \ol{\mu}, \\
F_c(t,x,y,m)&= b_1(t)y+(q(t) + \ol{q}(t))x -\ol{q}(t)  \ol{\mu}-ry+\ol{b}_1(t) \ol{\nu}, \\
F_g(t,x,y,m)&=b_1(t)y+(q(t) + \ol{q}(t))x -\ol{q}(t)  \ol{\mu}-ry,
\end{align*}
where $\mu$ and  $\nu$ are the first and second marginals of $m$ respectively. 

Applying Theorem~\ref{thm3.1}, \ref{thm3.2}, we can easily obtain the following two corollaries.

\begin{cor}
Suppose $b_1(t), b_2(t),q(t), \ol{q}(t), p(t)$ are bounded. Let $l, \iota, \eta, \xi$ be some positive constants. Then $\hat{\a}_t$ solves the mean field type control problem under either of the following:
\begin{enumerate}[(i)]
\item 
$|\ol{b}_1(t)| \leq l, \, -b_1(t) \geq l-\frac{r}{2}, \, \xi \geq p(t) \geq 2 \eta,\, q(t) \geq 2\iota, \, \ol{q}(t) \geq 0, \, |\ol{q}(t)| \leq l$ for all $t$, and \eqref{eq:condition1} holds.
\item $|\ol{b}_1(t)| \leq l,  \,  p(t) \geq 2 \eta,\, q(t)\geq 0, \, \ol{q}(t) \geq 0, \, |q(t)|+|\ol{q}(t)| \leq l$ for all $t$, and \eqref{eq:condition2} holds.
\end{enumerate}
\end{cor}

\begin{cor}
Suppose $b_1(t), b_2(t),q(t), \ol{q}(t), p(t)$ are bounded. Let $l, \iota, \eta, \xi$ be some positive constants. $\hat{\a}_t$ solves the mean field game  under either of the following two conditions:
\begin{enumerate}[(i)]
\item $|\ol{b}_1(t)| \leq l, \, -b_1(t) \geq l-\frac{r}{2}, \, \xi \geq p(t) \geq 2 \eta,\, q(t) \geq 2 \iota, \,  |\ol{q}(t)| \leq l$ for all $t$, and \eqref{eq:condition3} holds.
\item $|\ol{b}_1(t)| \leq l,  \,  p(t) \geq 2 \eta,\, q(t)+\ol{q}(t) \geq 0, \, |q(t)|+|\ol{q}(t)| \leq l$ for all $t$, and \eqref{eq:condition4} holds.
\end{enumerate}
\end{cor}

\begin{remark}
It is known that one can solve linear quadratic mean field games by Riccati equations, and thus the solution $Y_t$ is a linear transformation of $X_t$. As in \cite[Section 3.5]{MR3752669}, one may assume that $Y_t=\eta(t) X_t + \chi(t), \, Z_t=\eta(t) \sigma$, and it can be shown that $(\eta(t),\chi(t))$ solves 
\begin{align}\label{section4:eq1}
\begin{cases}
0=\dot{\eta}(t) -\eta(t)^2 \frac{b_2(t)^2}{p(t)}+\eta(t)\left( 2b_1(t)-r\right)+q(t)+\ol{q}(t), \\
0=\dot{\chi}(t)+\chi(t)\left(-\eta(t)\frac{b_2(t)^2}{p(t)}+b_1(t)-r\right)-\ol{q}(t) \ol{x}(t)+\eta(t) \ol{b}_1(t) \ol{x}(t), \quad \forall {t \geq 0},
\end{cases}
\end{align}
where $\ol{x}(t):=\E[X_t]$ together with $\ol{\eta}(t)$ is the solution to
\begin{align}\label{section4:eq2}
\begin{cases}
0=\dot{\ol{\eta}}(t)+\ol{\eta}(t)\left( 2b_1(t)+\ol{b}_1(t) -r\right) -\ol{\eta}(t)^2 \frac{b_2(t)^2}{p(t)}+q(t),  \\
\dot{\ol{x}}(t)=\left(b_1(t)+\ol{b}_1(t) -\ol{\eta}(t) \frac{b_2(t)^2}{p(t)} \right) \ol{x}(t),  \quad \forall t \geq 0,\\
\ol{x}(0)=\E[\xi].
\end{cases}
\end{align}
Both \eqref{section4:eq1} and \eqref{section4:eq2} are systems of infinite horizon ordinary differential equations, and we impose the growth condition $\int_0^{\infty} e^{-rt} \left( \ol{x}(t)^2+\chi(t)^2 \right) dt +\sup_t|\eta(t)| < \infty$, and that $\eta(t) \geq \frac{p(t)}{b_2(t)^2}(b_1(t)-r/2)$. 

When there exists a solution $(\eta(t),\chi(t),\ol{x}(t))$ to \eqref{section4:eq1}\eqref{section4:eq2}, it can be easily verified that $Y_t=\eta_t X_t +\chi_t, \, \E[X_t]=\ol{x}(t)$ solves \eqref{eq:mfg} and that $(X_t,Y_t) \in L^2_{-r}(0,\infty, \R^2)$. Therefore by the uniqueness result of MFG FBSDE \eqref{eq:mfg}, the solution to \eqref{section4:eq1}\eqref{section4:eq2} is also unique. The solvability of \eqref{section4:eq1} and \eqref{section4:eq2} is strongly connected with an equivalent deterministic linear quadratic optimal control problem, which is beyond the scope of this paper and  we refer to \cite[Section 3.5.1]{MR3752669}. Similarly, one can also write down ordinary equations for solutions to infinite horizon linear quadratic mean field control problems.

\end{remark}

\bibliographystyle{siam}
\bibliography{ref.bib}
\end{document}